\documentclass{article}
\usepackage[utf8]{inputenc}
\usepackage{pifont}
\usepackage{amssymb, amsmath, amsthm, amsfonts,enumerate}
\usepackage{amsmath,setspace,scalefnt}
\usepackage[usenames,dvipsnames,svgnames,table]{xcolor}
\usepackage{graphicx,tikz,caption,subcaption}
\usetikzlibrary{patterns}
\usetikzlibrary{shapes}
\usepackage{fullpage}
\usepackage{hyperref}
\usepackage{enumitem,verbatim}
\usetikzlibrary{decorations.pathmorphing}
\usepackage{subcaption}
\usepackage{pgfplots}
\usepackage{mathrsfs}
\usetikzlibrary{arrows}

\tikzset{snake it/.style={decorate, decoration=snake}}
\usetikzlibrary{calc}

\definecolor{gray}{rgb}{0.25, 0.25, 0.25}

\usetikzlibrary{decorations.markings}

\usetikzlibrary{matrix,arrows,decorations.pathmorphing}
\usetikzlibrary{positioning}
\usetikzlibrary{graphs} 
\usetikzlibrary{graphs.standard}

\newtheorem{theorem}{Theorem}[section]

\newtheorem{lemma}[theorem]{Lemma}
\newtheorem{fact}[theorem]{Fact}
\newtheorem{claim}[theorem]{Claim}

\newtheorem{observation}[theorem]{Observation}

\theoremstyle{definition}

\newtheorem{definition}[theorem]{Definition}

\title{Solution on strong partition of $2$-balanced regular multipartite tournaments}

\author{Jiangdong Ai\thanks{Corresponding author.  School of Mathematical Sciences and LPMC, Nankai University, Tianjin 300071, P.R.
China. Email: jd@nankai.edu.cn. },~ Fankang He\thanks{School of Mathematical Sciences and LPMC, Nankai University, Tianjin 300071, P.R.
China. Email: hefankang@mail.nankai.edu.cn.},~ Yihang Liu\thanks{School of Mathematical Sciences and LPMC, Nankai University, Tianjin 300071, P.R.
China. Email: liuyihang@mail.nankai.edu.cn.}}
\begin{document}

\maketitle
\begin{abstract}
We call a partition of a $c$-partite tournament into tournaments of order $c$ is strong if each tournament is strongly connected. The strong partition number denoted as $ST(r)$, represents the minimum integer $c'$ such that every regular $r$-balanced $c$-partite tournament has a strong partition with $c\geq c'$. 

Figueroa, Montellano-Ballesteros and Olsen showed the existence of $ST(r)$ for all $r\geq 2$ and proved that $5\leq ST(2)\leq 7$. In this note, we establish that $ST(2)=6$ and we also show the unique $2$-balanced $5$-partite tournament which has no strong partition.
\end{abstract}


\section{Introduction}
Notation follows \cite{MR2472389}. Here, we provide only the definitions relevant to this note. A digraph $D$ is \emph{strong} if there is a dipath between every ordered vertex pair of $V(D)$. A \emph{strong-$k$-partition} of a digraph $D =(V, A)$ is a partition $(V_1,..., V_k)$ of $V$ into $k$ non-empty disjoint sets such that the induced digraphs of these vertex subsets $D[V_1],..., D[V_k]$ are strong. For $c\geq 3$, a \emph{$c$-partite tournament} is a digraph obtained from a complete $c$-partite graph by orienting each edge. Let $G$ be a $c$-partite tournament with partite sets $\{ V_i\}_{i=1}^{c}$. We say $G$ is \emph{$r$-balanced} if all the partite sets are of the same size $r$. We denote a $r$-balanced $c$-partite tournament $G$ by $G_{r,c}$. A partition of $G_{r,c}$ into maximal tournaments is a spanning subgraph of $G_{r,c}$ formed by $r$ pairwise vertex-disjoint tournaments of order $c$. In the rest of this paper, a partition of $G_{r,c}$ always means a partition of $G_{r,c}$ into maximal tournaments.

Being strong is a basic property in digraphs. Partitioning a digraph into exactly $k$ strong subdigraphs has been extensively studied. Bang-Jensen, Cohen and Havet \cite{BangJensen2016partiton} proved that determining whether a digraph $D$ can be partitioned into $k$ strong subgraphs is $\mathcal{NP}$-hard. Reid \cite{Reid1985partition} proved that any $2$-strong $n$-vertex tournament with $n\geq 8$ can be partitioned into two strong subtournaments. Now we pay attention to partitioning a regular balanced multipartite tournament, instead of tournaments, into strong maximal tournaments. 
This problem also comes from Volkmann \cite{Volkmann2007}: How close to regular must a c-partite tournament be to guarantee a strong $c$-vertex tournament? And in a subsequent paper, Volkmann and Winzen \cite{Volkmann2008} conjectured that each vertex of a regular $c$-partite tournament with $c \geq 5$ partite sets is contained in a strong tournament of order $c$. This conjecture is proved by G. Xu, S. Li, H. Li and Q. Guo \cite{Xu2011} for the case when $c \geq 16$. Partitioning a regular balanced $c$-partite tournament into strong maximal tournaments is a stronger question, compared with the former conjecture, of whether this $c$-partite tournament has many disjoint strong $c$-vertex tournaments.

A partition of $G_{r,c}$ is called a \emph{strong partition} (st-partition for short) if every maximal tournament of this partition is strong, otherwise it is called a \emph{non-strong partition}. The \emph{st-partition number}, denoted by $ST(r)$, is the minimum integer $c'$ such that every $G_{r,c}$ has a st-partition if $c \geq c'$. This function was first defined by Figueroa, Montellano-Ballesteros and Olsen \cite{figueroa2023partition}, and they showed this function is well defined, i.e. such $c'$ exists for any integer $r\geq 2$. Furthermore, they find a rough range for $ST(2)$.

\begin{theorem}[\cite{figueroa2023partition}]\label{st27}
    $5 \leq ST(2) \leq 7$.
\end{theorem}

In this note, we determine the exact value of $ST(2)$.

\begin{theorem}\label{mainthm}
    $ST(2) = 6$.
\end{theorem}

\section{Preliminaries}

A $2$-balanced $5$-partite tournament $G_{2,5}$ is called \emph{nearly regular} if for any $v \in V(G_{2,5})$ we have $|d^{+}(v) - d^{-}(v)| \leq 2$. In this paper, all the $2$-balanced $5$-partite tournaments we considered are nearly regular, without special statements. In this section, we explore the structure of such $2$-balanced $5$-partite tournaments. Furthermore, we give an intuitive idea to construct a regular $2$-balanced $5$-partite tournament $G_{2,5}$ with no st-partition, which implies that $ST(2)\geq 6$.

\begin{definition}[control\ding{172}]
    Let $G_{2,c}$ be a regular $2$-balanced $c$-partite tournament with partite sets $V_{i} = \{v_{i,1}, v_{i,2}\}, i \in [c]$. For two distinct partite sets $V_i$ and $V_j$, We say that $V_i$ \emph{control\ding{172}} $V_j$, denoted by $V_i \stackrel{1}{\to} V_j$, if there are exactly two arcs from $V_i$ to $V_j$ and $|N^{+}(v_{i,1}) \cap V_j| = |N^{+}(v_{i,2}) \cap V_j| = |N^{+}(v_{j,1}) \cap V_i| = |N^{+}(v_{j,2}) \cap V_i| = 1$.
\end{definition}

\begin{definition}[control\ding{173}]
    Let $G_{2,c}$ be a regular $2$-balanced $c$-partite tournament with partite sets $V_{i} = \{v_{i,1}, v_{i,2}\}, i \in [c]$. For two distinct partite sets $V_i$ and $V_j$, we say that $V_i$ \emph{control\ding{173}} $V_j$, denoted by $V_i \stackrel{2}{\to} V_j$, if there are exactly two arcs from $V_i$ to $V_j$ and $|N^{+}(v_{j,1}) \cap V_i| \neq 1, |N^{+}(v_{j,2}) \cap V_i| \neq 1$.
\end{definition}

\begin{definition}[control\ding{174}, control\ding{175}]
    Let $G_{2,c}$ be a regular 2-balanced $c$-partite tournament with partite sets $V_{i} = \{v_{i,1}, v_{i,2}\}, i \in [c]$. For two distinct partite sets $V_i$ and $V_j$, we say that $V_i$ \emph{control\ding{174}} $V_j$, denoted by $V_i \stackrel{3}{\to} V_j$, if there are exactly three arcs from $V_i$ to $V_j$. We say that $V_i$ \emph{control\ding{175}} $V_j$, denoted by $V_i \stackrel{4}{\to} V_j$, if there are exactly four arcs from $V_i$ to $V_j$.
\end{definition}

We want to emphasize that throughout the entirety of the text, we may repeatedly substitute 'control' for 'controls', disregarding grammatical correctness. 

The above four relationships between two partite sets of $G_{2,c}$ are shown in Figure \ref{relation}. We say a vertex $u \in V$ is \emph{good to $V_i$} if $u \to v_{i,1}$ and $v_{i,2} \to u$ or $u \to v_{i,2}$ and $v_{i,1} \to u$. And $V_{i}$ is \emph{good to $V_{j}$}, if both $v_{i,1}$ and $v_{i,2}$ are good to $V_j$. Note that if $V_i$ is good to $V_j$ then $V_i \stackrel{1}{\to} V_j$ or $V_i \stackrel{2}{\to} V_j$. Let $v$ be a vertex in $D$, the \emph{minimum degree} of $v$ is $\delta(v) := \min \{d^{+}(v), d^{-}(v)\}$ and the \emph{minimum degree} of $D$ is $\delta(D) := \min_{v\in V(D)} \{\delta(v)\}$. A vertex $v$ is \emph{regular} if $d^{+}(v) = d^{-}(v)$, otherwise we say $v$ is \emph{irregular}. Note that if $V_i$ is good to $V_j$, then $V_i \stackrel{1}{\to} V_j$ or $V_i \stackrel{2}{\to} V_j$.

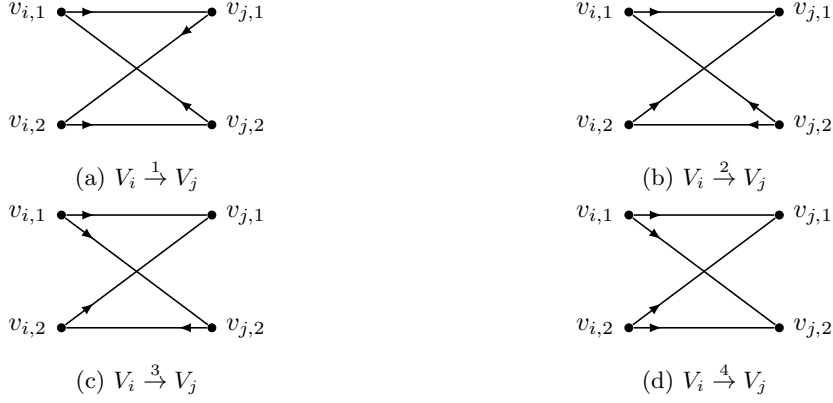
\begin{figure}[h]
    \centering
    \begin{subfigure}[t]{.45\textwidth}
        \centering
        \begin{tikzpicture}
        [
        corner/.style={ 
            circle,
            fill,   
            inner sep=1.2pt},	
        arrowline/.style={
            line width=0.6pt,
            postaction = decorate,
            decoration = {markings,
                mark = at position .2 with \arrow{latex}
            }
        },
        ]

        \node[corner,
        label = {left:$v_{i,1}$}] (v_{i,1}) at (-1,0){} ;
        
        \node[corner,
        label = {right:$v_{j,2}$}] (v_{j,2}) at (1,-1.5){} ;
        
        \node[corner,
        label = {left:$v_{i,2}$}] (v_{i,2}) at (-1,-1.5){} ;

        \node[corner,
        label = {right:$v_{j,1}$}] (v_{j,1}) at (1,0){} ;

        \draw [arrowline] (v_{j,2}) to (v_{i,1});
        \draw [arrowline] (v_{i,1}) to (v_{j,1});

        \draw [arrowline] (v_{i,2}) to (v_{j,2});
        \draw [arrowline] (v_{j,1}) to (v_{i,2});	
    \end{tikzpicture}
    \caption{$V_i \stackrel{1}{\to} V_j$}
    \end{subfigure}
    \begin{subfigure}[t]{.45\textwidth}
        \centering
        \begin{tikzpicture}
        [
        corner/.style={ 
            circle,
            fill,   
            inner sep=1.2pt},	
        arrowline/.style={
            line width=0.6pt,
            postaction = decorate,
            decoration = {markings,
                mark = at position .2 with \arrow{latex}
            }
        },
        ]

        \node[corner,
        label = {left:$v_{i,1}$}] (v_{1,1}) at (-1,0){} ;
        
        \node[corner,
        label = {right:$v_{j,2}$}] (v_{2,2}) at (1,-1.5){} ;
        
        \node[corner,
        label = {left:$v_{i,2}$}] (v_{1,2}) at (-1,-1.5){} ;

        \node[corner,
        label = {right:$v_{j,1}$}] (v_{2,1}) at (1,0){} ;

        \draw [arrowline] (v_{2,2}) to (v_{1,1});
        \draw [arrowline] (v_{1,1}) to (v_{2,1});

        \draw [arrowline] (v_{2,2}) to (v_{1,2});
        \draw [arrowline] (v_{1,2}) to (v_{2,1});	
    \end{tikzpicture}
    \caption{$V_i \stackrel{2}{\to} V_j$}
    \end{subfigure}
    \begin{subfigure}[t]{.45\textwidth}
        \centering
        \begin{tikzpicture}
        [
        corner/.style={ 
            circle,
            fill,   
            inner sep=1.2pt},	
        arrowline/.style={
            line width=0.6pt,
            postaction = decorate,
            decoration = {markings,
                mark = at position .2 with \arrow{latex}
            }
        },
        ]

        \node[corner,
        label = {left:$v_{i,1}$}] (v_{1,1}) at (-1,0){} ;
        
        \node[corner,
        label = {right:$v_{j,2}$}] (v_{2,2}) at (1,-1.5){} ;
        
        \node[corner,
        label = {left:$v_{i,2}$}] (v_{1,2}) at (-1,-1.5){} ;

        \node[corner,
        label = {right:$v_{j,1}$}] (v_{2,1}) at (1,0){} ;

        \draw [arrowline] (v_{1,1}) to (v_{2,2});
        \draw [arrowline] (v_{1,1}) to (v_{2,1});

        \draw [arrowline] (v_{2,2}) to (v_{1,2});
        \draw [arrowline] (v_{1,2}) to (v_{2,1});	
    \end{tikzpicture}
    \caption{$V_i \stackrel{3}{\to} V_j$}
    \end{subfigure}
    \begin{subfigure}[t]{.45\textwidth}
        \centering
        \begin{tikzpicture}
        [
        corner/.style={ 
            circle,
            fill,   
            inner sep=1.2pt},	
        arrowline/.style={
            line width=0.6pt,
            postaction = decorate,
            decoration = {markings,
                mark = at position .2 with \arrow{latex}
            }
        },
        ]

        \node[corner,
        label = {left:$v_{i,1}$}] (v_{i,1}) at (-1,0){} ;
        
        \node[corner,
        label = {right:$v_{j,2}$}] (v_{j,2}) at (1,-1.5){} ;
        
        \node[corner,
        label = {left:$v_{i,2}$}] (v_{i,2}) at (-1,-1.5){} ;

        \node[corner,
        label = {right:$v_{j,1}$}] (v_{j,1}) at (1,0){} ;

        \draw [arrowline] (v_{i,1}) to (v_{j,2});
        \draw [arrowline] (v_{i,1}) to (v_{j,1});

        \draw [arrowline] (v_{i,2}) to (v_{j,2});
        \draw [arrowline] (v_{i,2}) to (v_{j,1});	
    \end{tikzpicture}
    \caption{$V_i \stackrel{4}{\to} V_j$}
    \end{subfigure}
    \caption{Four relationships between two partite sets}
    \label{relation}
\end{figure}

\begin{lemma}[Folklore]\label{folk}
    Let T be a tournament of order c. If T is non-strong then $\delta(T)\leq \left \lfloor \frac{c-2}{4}\right \rfloor$.
\end{lemma}

Note that for a tournament $T$ on $5$ vertices, $T$ is non-strong if and only if $\delta(T) = 0$.

\begin{lemma}[\cite{figueroa2023partition}]\label{number}
    The number of partitions of any $G_{r,c}$ into maximal tournaments is $(r!)^{c-1}$
\end{lemma}

Given a nearly regular $2$-balanced $5$-partite tournament $G_{2,5}$ with partite sets $V_i = \{v_{i,1}, v_{i,2} \}$ for $i=[5]$. Let $N_1$ be the number of st-partitions of $G_{2,5}$. Let $N_2$ be the number of non-strong partitions of $G_{2,5}$. By Lemma \ref{number}, the number of partitions of $G_{2,5}$ into maximal tournaments is $16$, i.e. $N_1 + N_2 = 16$. Choose a vertex $u \in V(G_{2,5})$ and a partition $\tau$, let $T_{u,\tau}$ denote the maximal tournament containing $u$ under partition $\tau$. Let $n(u)$ be the number of partitions $\tau$ such that $\delta_{T_{u,\tau}}(u) = 0$. Then $n(u) \leq 2$. 
Let $n(V_i)$ be the number of non-strong partitions $\tau$ such that there exists a vertex $u \in V_{i}$ with $\delta_{T_{u,\tau}}(u) = 0$. Then $n(V_i) \leq n(v_{i,1}) + n(v_{i,2}) \leq 4$. Note that $N_2 \leq \sum_{i=1}^{5} n(V_i)$ since for each non-strong partition $\tau$ of $G_{2,5}$, there exists a maximal tournament $T$ under partition $\tau$ such that $T$ is non-strong, i.e. there is a vertex $u \in V(T)$ with $\delta_T(u) = 0$. If $\sum_{i=1}^{5} n(V_i) < 16$, then $G_{2,5}$ has a st-partition. The idea to construct a regular $5$-partite tournament $G_{2,5}$ with no st-partition is to make $\sum_{i=1}^{5} n(V_i)$ as large as possible. 

We say $V_i$ is \emph{regular} if both $v_{i,1}$ and $v_{i,2}$ are regular in $G_{2,5}$; $V_i$ is \emph{semi-regular} if there is exactly one regular vertex $u \in V_i$ in $G_{2,5}$; $V_i$ is \emph{irregular} if both $v_{i,1}$ and $v_{i,2}$ are irregular in $G_{2,5}$. We have the following useful observations.

\begin{observation}
    Let $G_{2,5}$ be a nearly regular $2$-balanced $5$-tournament. For a  vertex $u \in V_i$, $n(u) \in \{0, 2\}$. In particular, if $u$ is regular, $n(u) = 2$ if and only if $u$ is good to any other partite sets. Moreover, if $u$ is irregular, $n(u)=2$ if and only if $u$ is good to the other three partite sets.
\end{observation}

By studying a nearly regular $2$-balanced $5$-partite tournament $G_{2,5}$, we attain the following two lemmas.

\begin{lemma}\label{lemma rn4}
    Let $G_{2,5}$ be a nearly regular $2$-balanced $5$-tournament with partite sets $V_{i} = \{v_{i,1}, v_{i,2} \}, i\in [5]$. For any (semi-)regular partite set $V_i$ in $V(G_{2,5})$, if $n(V_i)=4$, then $V_i$ control\ding{173} some partite set $V_j$ where $j\ne i$.
\end{lemma}

\begin{proof}
   Assume the contrary that no partite set $V_j$ is controlled\ding{173} by $V_i$.
   
   When $V_i$ is regular, since $n(V_i) = 4$, we have $n(v_{i,1})=n(v_{i,2})=2$.
   Then $v_{i,1}$ and $v_{i,2}$ are good to other partite sets. So $V_i$ is good to other partite sets. Since there is no partite set being controlled\ding{173} by $V_i$, then $V_i$ control\ding{172} all other partite sets. For any partition $\tau$ of $G_{2,5}$, if  $\delta_{T_{v_{i,1},\tau}}(v_{i,1})=0$ , then $\delta_{T_{v_{i,2},\tau}}(v_{1,2})=0$. So $n(V_i)=2$. It's a contradiction.

   When $V_i$ is semi-regular, there is a vertex in $V_i$, say $v_{i,1}$, such that $|d^{+}(v_{i,1})-d^{-}(v_{i,1})| = 2$. W.l.o.g., let $d^{+}(v_{i,1})-d^{-}(v_{i,1}) = 2$. Since $n(V_i) = 4$, we have $n(v_{i,1})=n(v_{i,2})=2$. So $v_{i,2}$ is good to all other partite sets, and $v_{i,1}$ is good to all other partite sets except for one. Since there is no partite set being controlled\ding{173} by $V_i$, then $V_i$ control\ding{174} $V_j$ such that $N^{+}(v_{i,1}) \supseteq V_j$ and control\ding{172} $V_k$, $k\neq i,j$. But in this case,  $n(V_i) = 3$, it's a contradiction.
\end{proof}

\begin{lemma}\label{be controlled}
    Let $G_{2,5}$ be a nearly regular $2$-balanced $5$-tournament with partite sets $V_{i} = \{v_{i,1}, v_{i,2} \}, i\in [5]$. If a partite set $V_i$ in $V(G_{2,5})$ is controlled\ding{173} by two other partite sets, then $n(V_i) = 0$. In particular, if a regular part $V_i$ in $V(G_{2,5})$ is controlled\ding{173} by another partite set, then $n(V_i) = 0$.
\end{lemma}

\begin{proof}
    Let $V_j \stackrel{2}{\to} V_i$ and $V_{\ell} \stackrel{2}{\to} V_i$. We consider the following two cases.
    
    \textbf{Case 1: $N^{+}(V_j)\cap V_i = N^{+}(V_{\ell})\cap V_i$.}\\ W.l.o.g., let $N^{+}(V_j)\cap V_i = v_{i,1}$. Note that in this case, $N^{-}(V_j)\cap V_i = N^{-}(V_{\ell})\cap V_i = v_{i,2}$. As $G_{2,5}$ is nearly regular, there is another partite set $V_q(q\neq i,j,\ell)$ such that $V_q \subseteq N^{+}(v_{i,1})$. Then there is no partition $\tau$ such that $T_{v_{i,1},\tau}$ with $\delta(T_{v_{i,1},\tau})=0$, i.e $n(v_{i,1}) = 0$. Similar argument holds for $v_{i,2}$, we have $n(v_{i,2}) = 0$. Thus, we conclude that $n(V_i) = 0$.
    
    \textbf{Case 2: $N^{+}(V_j)\cap V_i \neq N^{+}(V_{\ell})\cap V_i$.}\\
    W.l.o.g., let $N^{+}(V_j)\cap V_i = v_{i,1}$. we have $V_j \subseteq N^{-}(v_{i,1}), V_{\ell} \subseteq N^{+}(v_{i,1})$ and $V_{\ell} \subseteq N^{-}(v_{i,2}), V_{j} \subseteq N^{+}(v_{i,2})$. Then $n(v_{i,1}) =0$ and $n(v_{i,2}) = 0$, which yields that $n(V_i) = 0$.

    Now let $V_i$ be regular, and $V_j \stackrel{2}{\to} V_i$. W.l.o.g., let $N^{+}(v_{i,1}) \supseteq V_{j}$ and $N^{-}(v_{i,2}) \supseteq V_{j}$. In order to make $v_{i,1}$ regular, there is another partite set $V_q(q\neq i,j)$ such that $V_j \subseteq N^{+}(v_{i,1})$. Then there is no partition $\tau$ such that $T_{v_{i,1},\tau}$ with $\delta_{T_{v_{i,1},\tau}} (v_{i,1})=0$, i.e $n(v_{i,1}) = 0$. Similar argument holds for $v_{i,2}$, we have $n(v_{i,2}) = 0$. Thus, $n(V_i) = 0$.
\end{proof}

We say a partition $\tau$ of $G_{2,5}$ is a \emph{good partition} if one of the two maximal tournaments is strong and the other is non-strong with exactly one vertex whose minimum degree is $0$ under $\tau$.

\begin{lemma}\label{good partition}
    Let $G_{2,5}$ be a nearly regular $2$-balanced $5$-tournament. If $G_{2,5}$ has no st-partition, then $G_{2,5}$ has at least 12 good partitions.
\end{lemma}

\begin{proof}
    By Lemma \ref{number}, the number of partitions of $G_{2,5}$ into maximal tournaments is $16$. Under these partitions, there are $32$ maximal tournaments. And the number of non-strong maximal tournaments is at most $\sum_{u\in V(G_{2,5})} n(u) \leq 10 \times 2 = 20$. Each partition contains at least one non-strong maximal tournament since $G_{2,5}$ has no st-partition. Therefore, $G_{2,5}$ has at least $12$ good partitions.
\end{proof}

\section{The unique regular $2$-balanced $5$-partite tournament with no st-partition}\label{3}

We first explore the structure of regular $2$-balanced $5$-partite tournaments with no st-partition. 

\begin{theorem}\label{st 5}
    Let $G_{2,5}$ be a regular $2$-balanced $5$-partite tournament with partite sets $V_{i} = \{v_{i,1}, v_{i,2} \}, i\in [5]$. If $G_{2,5}$ has no st-partition, then there is a partite set $V_i$ such that $V_j \stackrel{2}{\to} V_i$ for all $j \in [5]\setminus \{i\}$. Furthermore, the other four partite sets control\ding{172} each other.
\end{theorem}

\begin{proof}
    Since $G_{2,5}$ has no st-partition, $\sum_{i=1}^{5} n(V_i) \geq 16$. So, there is a partite set $V_j$ such that $n(V_j)=4$. By Lemma \ref{lemma rn4}, there exists a partite set $V_i$ such that $V_j \stackrel{2}{\to} V_i$. By the fact that $V_i$ is regular and Lemma \ref{be controlled}, $n(V_i)=0$. In order to make $\sum_{i=1}^5 n(V_{i}) \geq 16$, we have $n(V_j) = 4$ for all $j \in [5]$ and $j\ne i$. So, for any vertex $u\in {G_{2,5}}/{\{v_{i,1},v_{i,2}\}}$, $u$ is good to all other partite sets. Then $V_j$ is good to partite sets for all $j \in [5], j\neq i$. So $V_i$ is controlled\ding{173} by all other four partite sets. And the other four partite sets controll\ding{172} each other.
\end{proof}

Let $G_1$ and $G_2$ be the digraphs as shown in Figure \ref{suitableg42}.

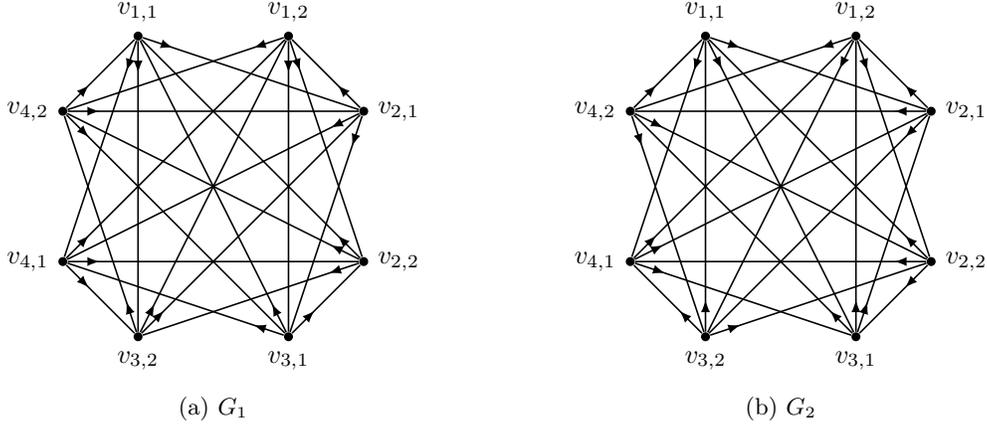
\begin{figure}[h]
    \centering
    \begin{subfigure}[t]{.45\textwidth}
        \centering
        \begin{tikzpicture}
    [
    corner/.style={ 
        circle,
        fill,   
        inner sep=1.2pt},	
    arrowline/.style={
        line width=0.6pt,
        postaction = decorate,
        decoration = {markings,
            mark = at position 12pt with \arrow{latex}
        }
    },
    ]

    \node[corner,
    label = {above:$v_{1,1}$}] (v11) at (-1,2){} ;
    
    \node[corner,
    label = {above:$v_{1,2}$}] (v12) at (1,2){} ;
    
    \node[corner,
    label = {right:$v_{2,1}$}] (v21) at (2,1){} ;

    \node[corner,
    label = {right:$v_{2,2}$}] (v22) at (2,-1){} ;

    \node[corner,
    label = {below:$v_{3,1}$}] (v31) at (1,-2){} ;

    \node[corner,
    label = {below:$v_{3,2}$}] (v32) at (-1,-2){} ;

    \node[corner,
    label = {left:$v_{4,1}$}] (v41) at (-2,-1){} ;

    \node[corner,
    label = {left:$v_{4,2}$}] (v42) at (-2,1){} ;
    
    \draw [arrowline] (v11) to (v41);
    \draw [arrowline] (v11) to (v21);
    \draw [arrowline] (v11) to (v32);
    
    \draw [arrowline] (v12) to (v42);
    \draw [arrowline] (v12) to (v22);
    \draw [arrowline] (v12) to (v31);
    
    \draw [arrowline] (v21) to (v31);
    \draw [arrowline] (v21) to (v41);
    \draw [arrowline] (v21) to (v12);

    \draw [arrowline] (v22) to (v32);
    \draw [arrowline] (v22) to (v42);
    \draw [arrowline] (v22) to (v11);

    \draw [arrowline] (v31) to (v11);
    \draw [arrowline] (v31) to (v41);
    \draw [arrowline] (v31) to (v22);

    \draw [arrowline] (v32) to (v12);
    \draw [arrowline] (v32) to (v42);
    \draw [arrowline] (v32) to (v21);

    \draw [arrowline] (v41) to (v12);
    \draw [arrowline] (v41) to (v22);
    \draw [arrowline] (v41) to (v32);

    \draw [arrowline] (v42) to (v11);
    \draw [arrowline] (v42) to (v21);
    \draw [arrowline] (v42) to (v31);
    \end{tikzpicture}
    \caption{$G_1$}
    \label{graph 42}
    \end{subfigure}
    \begin{subfigure}[t]{.45\textwidth}
        \centering
        \begin{tikzpicture}
    [
    corner/.style={ 
        circle,
        fill,   
        inner sep=1.2pt},	
    arrowline/.style={
        line width=0.6pt,
        postaction = decorate,
        decoration = {markings,
            mark = at position 12pt with \arrow{latex}
        }
    },
    ]

    \node[corner,
    label = {above:$v_{1,1}$}] (v11) at (-1,2){} ;
    
    \node[corner,
    label = {above:$v_{1,2}$}] (v12) at (1,2){} ;
    
    \node[corner,
    label = {right:$v_{2,1}$}] (v21) at (2,1){} ;

    \node[corner,
    label = {right:$v_{2,2}$}] (v22) at (2,-1){} ;

    \node[corner,
    label = {below:$v_{3,1}$}] (v31) at (1,-2){} ;

    \node[corner,
    label = {below:$v_{3,2}$}] (v32) at (-1,-2){} ;

    \node[corner,
    label = {left:$v_{4,1}$}] (v41) at (-2,-1){} ;

    \node[corner,
    label = {left:$v_{4,2}$}] (v42) at (-2,1){} ;
    
    \draw [arrowline] (v11) to (v21);
    \draw [arrowline] (v11) to (v31);
    \draw [arrowline] (v11) to (v41);
    
    \draw [arrowline] (v12) to (v22);
    \draw [arrowline] (v12) to (v32);
    \draw [arrowline] (v12) to (v42);
    
    \draw [arrowline] (v21) to (v12);
    \draw [arrowline] (v21) to (v42);
    \draw [arrowline] (v21) to (v32);

    \draw [arrowline] (v22) to (v11);
    \draw [arrowline] (v22) to (v41);
    \draw [arrowline] (v22) to (v31);

    \draw [arrowline] (v31) to (v21);
    \draw [arrowline] (v31) to (v12);
    \draw [arrowline] (v31) to (v42);

    \draw [arrowline] (v32) to (v22);
    \draw [arrowline] (v32) to (v11);
    \draw [arrowline] (v32) to (v41);

    \draw [arrowline] (v41) to (v12);
    \draw [arrowline] (v41) to (v21);
    \draw [arrowline] (v41) to (v31);

    \draw [arrowline] (v42) to (v11);
    \draw [arrowline] (v42) to (v22);
    \draw [arrowline] (v42) to (v32);
    \end{tikzpicture}
    \caption{$G_2$}
    \end{subfigure}
    \caption{Suitable $G_{2,4}$}
    \label{suitableg42}
\end{figure}

\begin{lemma}\label{g1g2}
    Let $G_{2,4}$ be a regular $2$-balanced $4$-partite tournament with partite sets $V_{i} = \{v_{i,1}, v_{i,2}\}, i\in [4]$. If all partite sets control\ding{172} each other, then $G_{2,4}$ is isomorphic to $G_1$ or $G_2$.
\end{lemma}

\begin{proof}
    In this proof, we say that $V_i$ control\ding{172} $V_j$ if $v_{i,1}$ dominates $v_{j,1}$ and $v_{i,2}$ dominates $v_{j,2}$, otherwise we say $V_j$ control\ding{172} $V_i$. Now we construct an auxiliary tournament $T(G_{2,4})$ of $G_{2,4}$ to record the relationships between partite sets of $G_{2,4}$. Let $V(T(G_{2,4})) = \{V_1, V_2, V_3, V_4\}$, and $V_i V_j$ is an arc of $T(G_{2,4})$ if and only if $V_i$ control\ding{172} $V_j$. 
    Note that the construction can be reversed: given a $4$-vertex tournament $T$ and a numbering of its vertex set $\{v_1,v_2,v_3,v_4\}$, there is a regular $2$-balanced $4$-partite tournament $G_{2,4}$ with partite sets $V_i, i\in [4]$ such that $T(G_{2,4}) = T$.
    Let $G'_{2,4}$ be the $2$-balanced $4$-partite tournaments obtained from $G_{2,4}$ by replacing partite set $V_{4} = \{v_{4,1}, v_{4,2}\}$ by $V'_{4} = \{v'_{4,1}, v'_{4,2}\}$ in which $v'_{4,1} = v_{4,2}$ and $v'_{4,2}= v_{4,1}$. Then $G'_{2,4}$ is isomorphic to $G_{2,4}$. And the auxiliary graph $T(G'_{2,4})$ can be obtained from  $T(G_{2,4})$ by reversing arcs incident to $V_4$. 
    
    There are exactly four kinds of $4$-vertex tournament up to isomorphism as shown in Figure \ref{T4}. Note that $T_1$ can be obtained from $T_2$ by reversing all arcs incident to $v_4$ and $T_3$ is isomorphic to the tournament obtained from $T_4$ by reversing all arcs incident to $v_3$. Thus, the corresponding $G_{2,4}$ of $T_1$ is isomorphic to the corresponding $G_{2,4}$ of $T_2$ and the corresponding $G_{2,4}$ of $T_3$ is isomorphic to the corresponding $G_{2,4}$ of $T_4$. 

    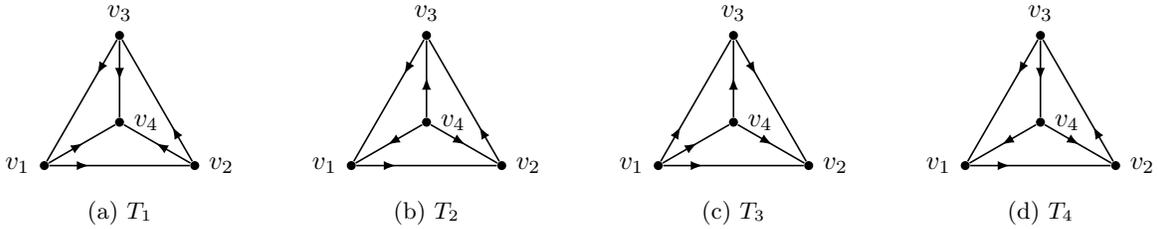
\begin{figure}[h]
    \centering
    \begin{subfigure}[t]{.24\textwidth}
        \centering
        \begin{tikzpicture}
        [
        corner/.style={ 
            circle,
            fill,   
            inner sep=1.2pt},	
        arrowline/.style={
            line width=0.6pt,
            postaction = decorate,
            decoration = {markings,
                mark = at position 15pt with \arrow{latex}
            }
        },
        ]

        \node[corner,
        label = {left:$v_1$}] (v1) at (-1,0){} ;
        
        \node[corner,
        label = {right:$v_2$}] (v2) at (1,0){} ;
        
        \node[corner,
        label = {above:$v_3$}] (v3) at (0,1.732){} ;

        \node[corner,
        label = {right:$v_4$}] (v4) at (0,0.577){} ;

        \draw [arrowline] (v1) to (v4);
        \draw [arrowline] (v2) to (v4);
        \draw [arrowline] (v3) to (v4);
        
        \draw [arrowline] (v1) to (v2);	
        \draw [arrowline] (v2) to (v3);	
        \draw [arrowline] (v3) to (v1);	
    \end{tikzpicture}
    \caption{$T_1$}
    \end{subfigure}
    \begin{subfigure}[t]{.24\textwidth}
        \centering
        \begin{tikzpicture}
        [
        corner/.style={ 
            circle,
            fill,   
            inner sep=1.2pt},	
        arrowline/.style={
            line width=0.6pt,
            postaction = decorate,
            decoration = {markings,
                mark = at position 15pt with \arrow{latex}
            }
        },
        ]

        \node[corner,
        label = {left:$v_1$}] (v1) at (-1,0){} ;
        
        \node[corner,
        label = {right:$v_2$}] (v2) at (1,0){} ;
        
        \node[corner,
        label = {above:$v_3$}] (v3) at (0,1.732){} ;

        \node[corner,
        label = {right:$v_4$}] (v4) at (0,0.577){} ;

        \draw [arrowline] (v4) to (v1);
        \draw [arrowline] (v4) to (v2);
        \draw [arrowline] (v4) to (v3);
        
        \draw [arrowline] (v1) to (v2);	
        \draw [arrowline] (v2) to (v3);	
        \draw [arrowline] (v3) to (v1);	
    \end{tikzpicture}
    \caption{$T_2$}
    \end{subfigure}
    \begin{subfigure}[t]{.24\textwidth}
        \centering
        \begin{tikzpicture}
        [
        corner/.style={ 
            circle,
            fill,   
            inner sep=1.2pt},	
        arrowline/.style={
            line width=0.6pt,
            postaction = decorate,
            decoration = {markings,
                mark = at position 15pt with \arrow{latex}
            }
        },
        ]

        \node[corner,
        label = {left:$v_1$}] (v1) at (-1,0){} ;
        
        \node[corner,
        label = {right:$v_2$}] (v2) at (1,0){} ;
        
        \node[corner,
        label = {above:$v_3$}] (v3) at (0,1.732){} ;

        \node[corner,
        label = {right:$v_4$}] (v4) at (0,0.577){} ;

        \draw [arrowline] (v1) to (v4);
        \draw [arrowline] (v4) to (v2);
        \draw [arrowline] (v4) to (v3);
        
        \draw [arrowline] (v1) to (v2);	
        \draw [arrowline] (v3) to (v2);	
        \draw [arrowline] (v1) to (v3);	
    \end{tikzpicture}
    \caption{$T_3$}
    \end{subfigure}
    \begin{subfigure}[t]{.24\textwidth}
        \centering
        \begin{tikzpicture}
        [
        corner/.style={ 
            circle,
            fill,   
            inner sep=1.2pt},	
        arrowline/.style={
            line width=0.6pt,
            postaction = decorate,
            decoration = {markings,
                mark = at position 15pt with \arrow{latex}
            }
        },
        ]

        \node[corner,
        label = {left:$v_1$}] (v1) at (-1,0){} ;
        
        \node[corner,
        label = {right:$v_2$}] (v2) at (1,0){} ;
        
        \node[corner,
        label = {above:$v_3$}] (v3) at (0,1.732){} ;

        \node[corner,
        label = {right:$v_4$}] (v4) at (0,0.577){} ;

        \draw [arrowline] (v4) to (v1);
        \draw [arrowline] (v4) to (v2);
        \draw [arrowline] (v3) to (v4);
        
        \draw [arrowline] (v1) to (v2);	
        \draw [arrowline] (v2) to (v3);	
        \draw [arrowline] (v3) to (v1);	
    \end{tikzpicture}
    \caption{$T_4$}
    \end{subfigure}
    \caption{Four kinds of 4-vertex tournaments}
    \label{T4}
\end{figure}

    The corresponding $G_{2,4}$ of $T_1$ is $G_1$ and the corresponding $G_{2,4}$ of $T_3$ is $G_2$. This means that such $G_{2,4}$ is isomorphic to $G_1$ or $G_2$.
\end{proof}

\begin{theorem}\label{H0}
    There is a unique regular 2-balanced 5-partite tournament $H_0$ with no st-partition.
\end{theorem}

\begin{proof}
    Suppose the partite sets of $G_{2,5}$ are $V_{i} = \{v_{i,1}, v_{i,2}\}, i\in [5]$. By Lemma \ref{st 5}, we know that there is a partite set of $G_{2,5}$, say $V_5$, being controlled\ding{173} by some other partite sets. As $V_5$ is regular, the arcs incident to $V_5$ would be as shown in Figure \ref{V5}.

    \begin{figure}[h]
        \centering
        \begin{tikzpicture}
    [
    corner/.style={ 
        circle,
        fill,   
        inner sep=1.2pt},	
    arrowline/.style={
        line width=0.6pt,
        postaction = decorate,
        decoration = {markings,
            mark = at position 0.3 with \arrow{latex}
        }
    },
    ]

    \node[corner,
    label = {above:$v_{5,1}$}] (v51) at (-1.5,2){} ;
    
    \node[corner,
    label = {above:$v_{5,2}$}] (v52) at (1.5,2){} ;
    
    \node[corner,
    label = {below:$v_{i,1}$}] (v1) at (-5,0){} ;

    \node[corner,
    label = {below:$v_{i,2}$}] (v2) at (-4,0){} ;

    \node[corner,
    label = {below:$v_{j,1}$}] (v3) at (-2,0){} ;

    \node[corner,
    label = {below:$v_{j,2}$}] (v4) at (-1,0){} ;

    \node[corner,
    label = {below:$v_{k,1}$}] (v5) at (1,0){} ;

    \node[corner,
    label = {below:$v_{k,2}$}] (v6) at (2,0){} ;
    
    \node[corner,
    label = {below:$v_{\ell,2}$}] (v7) at (4,0){} ;

    \node[corner,
    label = {below:$v_{\ell,2}$}] (v8) at (5,0){} ;
    
    \draw [arrowline] (v51) to (v1);
    \draw [arrowline] (v51) to (v2);
    \draw [arrowline] (v51) to (v3);
    \draw [arrowline] (v51) to (v4);

    \draw [arrowline] (v52) to (v5);
    \draw [arrowline] (v52) to (v6);
    \draw [arrowline] (v52) to (v7);
    \draw [arrowline] (v52) to (v8);

    \draw [arrowline, blue] (v1) to (v52);
    \draw [arrowline, blue] (v2) to (v52);
    \draw [arrowline, blue] (v3) to (v52);
    \draw [arrowline, blue] (v4) to (v52);

    \draw [arrowline, blue] (v5) to (v51);
    \draw [arrowline, blue] (v6) to (v51);
    \draw [arrowline, blue] (v7) to (v51);
    \draw [arrowline, blue] (v8) to (v51);
    \end{tikzpicture}
        \caption{Arcs incident to $V_5$}
        \label{V5}
    \end{figure}
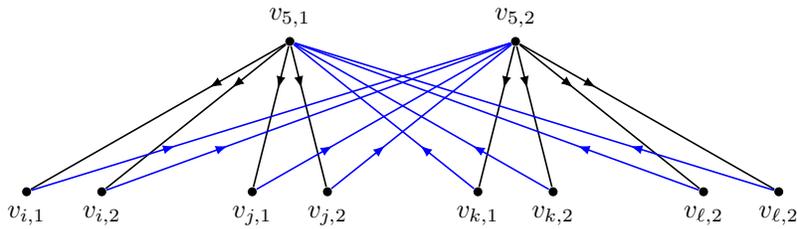
    
    Let $G := G_{2,5} \setminus V_5$. Then $G$ is a regular $2$-balanced $4$-partite tournament and the partite sets of $G$ control\ding{172} each other. By Lemma \ref{g1g2}, we have that $G$ is isomorphic to $G_1$ or $G_2$. 

    But for $G_2$, there is a partition $V_a = \{v_{1,1}, v_{2,2}, v_{3,1}, v_{4,2}\}, V_b = \{v_{1,2}, v_{2,1}, v_{3,2}, v_{4,1}\}$ such that both $D[V_a]$ and $D[V_b]$ are strong. Then $V_a = \{v_{1,1}, v_{2,2}, v_{3,1}, v_{4,2}, v_{5,1}\}, V_b = \{v_{1,2}, v_{2,1}, v_{3,2}, v_{4,1}, v_{5,2}\}$ would be a st-partition of $G_{2,5}$. Now let $H_0$ be the joint of $G_1$ and partite sets $V_5 = \{v_{5,1}, v_{5,2}\}$ such that $V_5$ is controlled\ding{173} by all other partite sets. Note such $H_0$ is unique up to isomorphism. Any partition of $H_0$ contains exactly one non-strong maximal tournament. Therefore, $H_0$ is the unique regular $2$-balanced $5$-partite tournament with no st-partition, up to isomorphism.
\end{proof}

Thus, we have the following lemma.

\begin{lemma}\label{st25}
    $ST(2) \geq 6$.
\end{lemma}

\section{Every regular $2$-balanced $6$-partite tournament has a st-partition}

Given a regular $2$-balanced $6$-partite tournament $G_{2,6}$ with partite sets $V_i = \{v_{i,1}, v_{i,2} \}$ for $i=[6]$, observe that after deleting $V_6$ from $G_{2,6}$, we get a nearly regular $2$-balanced $5$-partite tournament $G_{2,5}$ with partite sets $V_i = \{v_{i,1}, v_{i,2} \}, i=[5]$. The following lemmas are useful in exploring the structure of $G_{2,6}$. 

\begin{lemma}\label{control3}
    For any irregular part $V_i$ in $V(G_{2,5})$ where $i\in [5]$, if $n(V_i) = 4$, then $V_i$  control\ding{173} another partite set $V_j$ or $V_i$ is controlled\ding{173} by $V_6$.
\end{lemma}

\begin{proof}
         
    W.l.o.g., let $V_1$ be an irregular partite set in $G_{2,5}$. If $d^+_{G_{2,5}}(v_{1,1})-d^-_{G_{2,5}}(v_{1,1})=2$ and $d^+_{G_{2,5}}(v_{1,2})-d^-_{G_{2,5}}(v_{1,2})=-2$, we have that $V_1$ is controlled\ding{173} by $V_6$. So, we only need to consider the case that the in-degree( or out-degree) of $v_{1,1}$ and $v_{1,2}$ are $5$. For convenience, we assume that $d^+_{G_{2,5}}(v_{1,1})=d^+_{G_{2,5}}(v_{1,2})=5$. Observe that $n(v_{1,i})=2$ if and only if $v_{1,i}$ is good to exactly three partite sets. If $V_1$ is good to $V_j$, $V_1$ will control\ding{172} or control\ding{173} $V_j,j \in [5]$ and $i\ne j$. Then we consider the following two cases. 

    \textbf{Case 1: $v_{1,1}$ and $v_{1,2}$ are good to the same part.}\\
    If $V_1$ doesn't control\ding{173} any partite set, $V_1$ control\ding{175} a partite set and control\ding{172} other three partite sets. Observe that $n(V_1)$ is $2$ in this case, it's a contradiction.   

    \textbf{Case 2: $v_{1,1}$ and $v_{1,2}$ are good to different partite sets.}\\
    So, $V_1$ control\ding{174} two different partite sets. If $V_1$ doesn't control\ding{173} any partite set, then $V_1$ will control\ding{172} other two partite sets. Observe that $n(V_1)=2$ in this case, it's also a contradiction. 

    Therefore, $v_i$ control\ding{173} another part $V_j$ or $V_i$ is controlled\ding{173} by $V_6$.
    
\end{proof}

\begin{lemma}\label{subG25}
    Let $G_{2,6}$ be a regular 2-balanced $6$-partite tournament. If $G_{2,6}$ has no st-partition, then it contains a $2$-balanced $5$-partite tournament $G_{2,5}$ with no st-partition as a subgraph.
\end{lemma}

\begin{proof}
     Assume the contrary that for any $2$-balanced 5-partite tournament $G_{2,5}$ in $G_{2,6}$, there exists a st-partition of $G_{2,5}$. We claim that the relations between any two partite sets are control\ding{172} or control\ding{174} in $G_{2,6}$. Let $G_i$ be the $2$-balanced 5-partite tournament we get after deleting $V_i$ from $G_{2,6}$. By assumption, $G_i$ has a st-partition into $T_1$ and $T_2$. Then there exists a vertex $u \in V_i$ such that $N^{+}_{G_{2,6}}(u) = V(T_1), N^{-}_{G_{2,6}}(u) = V(T_2)$ or $N^{+}_{G_{2,6}}(u) = V(T_2), N^{-}_{G_{2,6}}(u) = V(T_1)$. Otherwise we can add $v_{i,1}$ to $T_1$ and $v_{i,2}$ to $T_2$, then we get a st-partition of $G_{2,6}$, a contradiction. Note that such vertex $u \in V_i$ is a vertex that is good to each other partite set in $G_{2,6}$, we call such vertex a good vertex; and if both vertices of $V_i$ are good vertices, we call $V_i$ a good partite set. Then for every partite set, there is a good vertex. Thus, a partite set neither control\ding{173} nor control\ding{175} other partite sets, otherwise we can not find a good vertex in each partite set. Moreover, a partite set can control\ding{174} at most two partite sets. Now, we need the following claim.

     \begin{claim}\label{claimstG26}
         Let $G_{2,5} := G_{2,6} \setminus V_{\ell}$, then $\sum_{i\neq \ell} n_{G_{2,5}}(V_i) \geq 14$. Moreover, if $\sum_{i \neq \ell} n_{G_{2,5}}(V_i) =14$ and $G_{2,6}$ has no st-partition, then $V_{\ell}$ is good to each other partite set in $G_{2,6}$.
     \end{claim}

     \begin{proof}[Proof of Claim \ref{claimstG26}]
         Suppose that $\tau$ is a st-partition of $G_{2,5}$ and $\tau$ partitions $G_{2,5}$ into $T^{\tau}_1$ and $T^{\tau}_2$. Since $G_{2,6}$ has no st-partition, then there is vertex $u \in V_{\ell}$ such that $N^{+}(u) = V(T^{\tau}_1), N^{-}(u) = V(T^{\tau}_2)$ or $N^{-}(u) = V(T^{\tau}_1), N^{+}(u) = V(T^{\tau}_2)$. This means that $u$ is a good vertex. Now assume that $\sum_{i\neq \ell} n_{G_{2,5}}(V_i) \leq 13$, then there are at least three st-partitions of $G_{2,5}$. This implies that there are at least three different vertices in $V_{\ell}$ since for different st-partitions $\tau$ and $\tau'$, we have $V(T^{\tau}_1) \neq V(T^{\tau'}_1), V(T^{\tau}_2) \neq V(T^{\tau'}_2)$ and $V(T^{\tau}_1) \neq V(T^{\tau'}_2)$. While $|V_{\ell}| = 2$, and so it is a contradiction. Thus, $\sum_{i\neq \ell} n_{G_{2,5}}(V_i) \geq 14$. And if $\sum_{i\neq \ell} n_{G_{2,5}}(V_i) = 14$, then there are two st-partitions of $G_{2,5}$. So both two vertices of $V_{\ell}$ are good vertices. 
     \end{proof}
   Then, we consider the following two cases:
   
     \textbf{Case 1: There are no partite sets $V_i$ and $V_j$ such that $V_i \stackrel{3}{\to} V_j$ in $G_{2,6}$.}\\
     Let $G := D[\bigcup_{1\leq i \leq 5} V_i]$, then $G$ is a regular $2$-balanced $5$-partite tournament. Note that $n(V_i) = 2$ for all $i \in [5]$ since for any two partite sets $V_1$ and $V_2$, we have $V_i \stackrel{1}{\to} V_j$. Then $\sum_{i=1}^5 n_{G}(V_i) = 10$, a contradiction to Claim \ref{claimstG26}.

     \textbf{Case 2: There are partite sets $V_1$ and $V_2$ such that $V_1 \stackrel{3}{\to} V_2$ in $G_{2,6}$.}\\
     Let $D$ be a digraph. Now we define the \emph{co-out-neighborhood} of a vertex subset $V_0$ is $N_{co}^{+}(V_0) = \cap_{v\in V_0} N^{+}(v)$, and the \emph{co-in-neighborhood} of a vertex subset $V_0$ is $N_{co}^{-}(V_0) = \cap_{v\in V_0} N^{-}(v)$. 
     
     As $V_1$ is regular, $V_1$ is controlled\ding{174} by another partite set, say $V_3$, such that $N^{-}_{co}(V_2) \cap V_1 = N^{+}_{co}(V_3) \cap V_3$. And as $V_3$ is regular, $V_3$ is controlled\ding{174} by another partite set $V_j, j\neq 1,3$, such that $N^{-}_{co}(V_1) \cap V_3 = N^{+}_{co}(V_j) \cap V_3$. 
     
     \textbf{Subcase 2.1: $j = 2$.}\\
     If $N^{-}_{co}(V_3) \cap V_2 \neq N^{+}_{co}(V_1) \cap V_2$, then as $V_2$ is regular in $G_{2,6}$, there must be another two partite sets $V_\ell$ and $V_{\ell'}$ such that $V_2$ is controlled\ding{174} by $V_\ell$ and $V_{2} $ control\ding{174} $V_{\ell'}$. Besides, $N^{-}_{co}(V_3) \cap V_2 = N^{+}_{co}(V_\ell) \cap V_2$ and $N^{-}_{co}(V_{\ell'}) \cap V_2 = N^{+}_{co}(V_1) \cap V_2$. Then there remains a partite set $V_k$ such that $k\not\in \{1,2,3, \ell, \ell'\}$. Let $G := G_{2,6} \setminus V_k$, then $n_G(V_1) = 0, n_G(V_2) \leq 2$ and $n_G(V_3) \leq 2$. As a result, $\sum_{i\neq k}n_G(V_i) \leq 0 + 2\times 2 + 2\times 4 = 12$, a contradiction to Claim \ref{claimstG26}. Now we consider the case when $N^{-}_{co}(V_3) \cap V_2 = N^{+}_{co}(V_1) \cap V_2$. If for any $V_j ,j \in \{4,5,6\}$, $V_j$ does not control\ding{174} any other partite sets, let $G := D[\bigcup_{1\leq i \leq 5} V_i]$. Then $n_G (V_i) \leq 2$ for all $i \in [5]$. As a result, $\sum_{i}^5 n_G(V_i) \leq 10$, a contradiction to Claim \ref{claimstG26}. If there is a partite set $V_k ,k \in \{4,5,6\}$ such that $V_k$ control\ding{174} some partite set $V_\ell, \ell \neq k$, let $G := G_{2,6} \setminus V_k$, then $n_G(V_1) \leq 2, n_G(V_2) \leq 2$ and $n_G(V_3) \leq 2$. So $\sum_{i \neq k} n_G(V_i) \leq 14$. Observe that $V_k$ is not a good partite set in $G_{2,6}$ since $V_k$ control\ding{174} $V_\ell$. This contradicts to Claim \ref{claimstG26}.

     \textbf{Subcase 2.2. $j \geq 4$.}\\ W.l.o.g., let $j = 4$. Again, because $V_4$ is regular, $V_4$ is controlled\ding{174} by another partite set $V_{\ell}, \ell \neq 3$, such that $N^{-}_{co}(V_1) \cap V_3 = N^{+}_{co}(V_j) \cap V_3$. If $\ell \neq 2$, let $G := G_{2,6} \setminus V_2$. Then $n_G (V_1),n_G (V_3), n_G (V_4) \leq 2$, and so $\sum_{i \neq 2} n_G(V_i) \leq 3 \times 2 + 2 \times 4 = 14$. Observe that $V_2$ is not good in $G_{2,6}$, since $V_2$ is controlled\ding{174} by $V_1$. It is a contradiction to Claim \ref{claimstG26}. Now it remains to consider the case when $\ell = 2$. If $N^{-}_{co}(V_4) \cap V_2 = N^{+}_{co}(V_1) \cap V_2$, let $G := D[\bigcup_{1\leq i \leq 5} V_i]$. Then $n_G(V_i) \leq 2$ for all $i \in [4]$ and $n_G(V_5) \leq 4$. Therefore $\sum_{i=1}^5 n_G(V_i) \leq 12$, a contradiction to Claim \ref{claimstG26}.
\end{proof}

\begin{lemma}\label{G26structure}
    Let $G_{2,6}$ be a regular $2$-balanced $6$-partite tournament. If $G_{2,6}$ has no st-partition, there are two partite sets $V_{i},V_{j}$ in $G_{2,6}$ such that $V_i$  control\ding{173} $V_j$.
\end{lemma}

\begin{proof}
    By lemma \ref{subG25}, there is a $2$-balanced $5$-partite graph $G_{2,5} \subseteq G_{2,6}$ and this $G_{2,5}$ has no st-partition. Thus, the sum of $n(V_i)$ over all partite sets $V_i$ of $G_{2,5}$ is at least $16$. By the Pigeon Hole Principle, there is a partite set $V_i$ of $G_{2,5}$ with $n(V_i) = 4$. Then by Lemma \ref{lemma rn4} and Lemma \ref{control3}, we get the desired result.
\end{proof}

\begin{theorem}\label{st26}
    There is no regular $2$-balanced $6$-partite tournament with no st-partition.
\end{theorem}

\begin{proof}
    Assume the contrary that there is a regular $2$-balanced $6$-tournament $G_{2,6}$ with no st-partition. By Lemma \ref{G26structure}, there is a partite set, say $V_5$, control\ding{173} another partite set, say $V_6$. Then, there exists a vertex $u \in V_6$ such that $v_{5,1}\to u$ and $v_{5,2}\to u$. W.l.o.g., let $u$ be $v_{6,1}$. Since $v_{6,1}$ is regular in $G_{2,6}$, there is another partite set $V_{j}, j\neq 5,6$ such that $v_{6,1}\to v_{j,1}$ and $v_{6,1}\to v_{j,2}$.

    For the other vertex $v_{6,2} \in V_6$, we have $v_{6,2} \to v_{5,1}$ and $v_{6,2} \to v_{5,2}$. As $v_{6,2}$ is regular, there is another partite set $V_{j'}, j'\neq 5,6$ such that $v_{j',1} \to v_{6,2}$ and $v_{j',2} \to v_{6,2}$. Deleting $V_6$ from $G_{2,6}$, we get a nearly regular $2$-balanced $5$-tournament, denoted by $G_{2,5}$. Suppose $\tau$ is a st-partition of $G_{2,5}$, then $\tau$ partitions $G_{2,5}$ into two maximal tournaments, denoted by $T^{\tau}_1$ and $T^{\tau}_2$ respectively. Note that $G_{2,5}$ has no st-partiton, otherwise $V(G_{2,6}) = W_1 \cup W_2$, in which $W_1 = \{v_{6,1}\} \cup V(T^{\tau}_1)$ and $W_2 = \{v_{6,2}\} \cup V(T^{\tau}_2)$, would be a st-partiton for $G_{2,6}$.

    By Lemma \ref{good partition}, there are at least $12$ good partitions for $G_{2,5}$. For every good partition $\tau$ for $G_{2,5}$, $\tau$ partitions $G_{2,5}$ into one strong maximal tournament, denoted by $T^{\tau}_1$, and one non-strong maximal tournament, denoted by $T^{\tau}_2$, with exactly one vertex $v_1 \in V(T^{\tau}_2)$ having minimum degree $0$ in $V(T^{\tau}_2)$.

    \begin{claim}\label{regularclaim}
        Such vertex $v_1$ is regular in $G_{2,5}$.
    \end{claim}

    \begin{proof}[Proof of Claim \ref{regularclaim}]
        Assume the contrary that $v_1$ is irregular in $G_{2,5}$, w.l.o.g., let $d^{+}_{T^{\tau}_2}(v_1) = 4$ and name the other four vertices in $T^{\tau}_2$ as $v_2, v_3, v_4$ and $v_5$ respectively. As $v_1$ is regular in $G_{2,6}$, we have that $v_{6,1} \to v_1$ and $v_{6,2} \to v_1$. Since $V_6$ is controlled\ding{173} by $V_5$, neither $v_{6,1}$ nor $v_{6,2}$ is good in $G_{2,6}$. Since $v_{5,1} \to v_{6,1}$ and $v_{5,2} \to v_{6,1}$, there is a vertex $V(T_2^{\tau})$, say $v_2$, such that $v_2 \to v_{6,1}$. 
        If there is a vertex in $V(T^{\tau}_2 \setminus \{v_1, v_2\})$, say $v_3$, such that $v_3 \to v_2$. Then there is a cycle $v_1 v_2 v_{6,1}$. Suppose $D[v_3 v_4 v_5]$ forms a cycle. Note that $v_{1} \to v_5$ and $v_3 \to v_2$. Then $D[v_{6,1} \cup V(T^{\tau}_2)]$ is a strong tournament. Recall that $T_{1}^{\tau}$ is strong, $D[v_{6,2} \cup V(T^{\tau}_1)]$ is also strong. This yields a st-partition of $G_{2,6}$, a contradiction. Suppose $D[v_3 v_4 v_5]$ forms a transitive triangle and let $u$ and $u'$ be the sink and source of $D[v_3 v_4 v_5]$ respectively. Since $\tau$ is a good partition, then $\delta_{T^{\tau}_2}(v_i) \geq 1$ for $2 \leq i \leq 5$. Then there is an arc from $u$ to $\{v_1, v_2, v_{6,1} \}$. Together with $v_1 \to u'$, we have $D[v_{6,1} \cup V(T^{\tau}_2)]$ is a strong tournament which will lead to that $G_{2,6}$ has a st-partition as before.
        As a result, there is no vertex $v \in \{v_3, v_4, v_5\}$ such that $v \to v_2$. Since $\delta_{T^{\tau}_2}(v_i) \geq 1$ for all $2\leq i \leq 5$, then $v_3 v_4 v_5$ would be a cycle in this case. And note that $v_{6,1} v_1 v_2$ is also a cycle. Then all arcs between these two cycles should be from $\{v_3, v_4, v_5 \}$ to $\{v_{6,1}, v_1, v_2\}$, otherwise $D[v_{6,1} \cup V(T^{\tau}_2)]$ is a strong tournament which will lead to that $G_{2,6}$ has a st-partition as before. Then consider the neighborhood of $v_{6,1}$, we have $v_{6,1} \to v_1, v_2 \to v_{6,1}, v_{6,1} \to v_3, v_{6,1} \to v_4$ and $v_{6,1} \to v_5$, as shown in Figure \ref{4in1out}. By similar arguments for $v_{6,2}$, we have $v_{6,2} \to v_1, v_2 \to v_{6,2}, v_{6,2} \to v_3, v_{6,2} \to v_4$ and $v_{6,2} \to v_5$, as shown in Figure \ref{4in1out}. This means that for any partite set $V_i, i\in [5]$, there is a vertex $u \in V_i$ such that $N^{+}(u) \supseteq V_6$ and $N^{-}(u) \supseteq V_6$. Then no partite set $V_i ,i\in [5]$ control\ding{173} $V_6$, otherwise both $v_{i,1}$ and $v_{i,2}$ are good to $V_6$. This contradicts the fact that $V_6$ is controlled\ding{173} by another partite set.

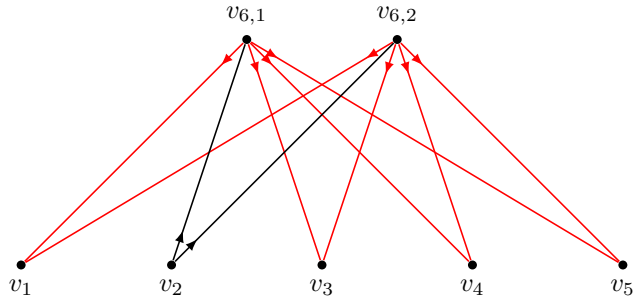
\begin{figure}[!htpb]
    \centering
    \begin{tikzpicture}
    [
    corner/.style={ 
        circle,
        fill,   
        inner sep=1.2pt},	
    arrowline/.style={
        line width=0.6pt,
        postaction = decorate,
        decoration = {markings,
            mark = at position 12pt with \arrow{latex}
        }
    },
    ]

    \node[corner,
    label = {above:$v_{6,1}$}] (v61) at (-1,3){} ;
    
    \node[corner,
    label = {above:$v_{6,2}$}] (v62) at (1,3){} ;
    
    \node[corner,
    label = {below:$v_1$}] (v1) at (-4,0){} ;

    \node[corner,
    label = {below:$v_2$}] (v2) at (-2,0){} ;

    \node[corner,
    label = {below:$v_3$}] (v3) at (0,0){} ;

    \node[corner,
    label = {below:$v_4$}] (v4) at (2,0){} ;

    \node[corner,
    label = {below:$v_5$}] (v5) at (4,0){} ;

    \draw [arrowline, red] (v61) to (v1);
    \draw [arrowline, red] (v61) to (v3);
    \draw [arrowline, red] (v61) to (v4);
    \draw [arrowline, red] (v61) to (v5);
    
    \draw [arrowline, red] (v62) to (v1);
    \draw [arrowline, red] (v62) to (v3);
    \draw [arrowline, red] (v62) to (v4);
    \draw [arrowline, red] (v62) to (v5);
    
    \draw [arrowline] (v2) to (v61);
    \draw [arrowline] (v2) to (v62);
    \end{tikzpicture}
    \caption{Arcs between $V_6$ and $V(T^{\tau}_2)$}
    \label{4in1out}
    \label{fig:1}
\end{figure}
    \end{proof}

Now, we modify the function $n(\cdot)$ to $n'(\cdot)$ to obtain a better approximation of the number of non-strong partitions $N_2$ for $G_{2,5}$. First, we classify the non-strong partitions of $G_{2,5}$ into two categories. Let $\tau$ be a non-strong partition of $G_{2,5}$. For a fixed $V_i$, we say $\tau$ is a \emph{useful partition} if either there is a regular vertex $u\in V_i$ and $\delta_{T_{u,\tau}}(u)=0$ or if $V_i$ is irregular and $\delta_{T_{v_{i,1}, \tau}}(v_{i,1}) = \delta_{T_{v_{i,2}, \tau}}(v_{i,2}) = 0$;
we say $\tau$ is a \emph{useless partition} if there is exactly one vertex $v\in V_i$ such that $\delta_{T_{v,\tau}}(v)=0$ and $v$ is irregular. Let $n_1(V_i)$ be the number of useful partitions for $V_i$, and $n_2(V_i)$ be the number of useless partitions for $V_i$. Note that $n(V_i) = n_1(V_i) + n_2(V_i)$. Define $n'(V_i) = n_1(V_i) + \frac{1}{2}n_2(V_i)$.

\begin{claim}\label{N2upper}
    $\sum_{i=1}^5 n'(V_i) \geq 16$.
\end{claim}

\begin{proof}[Proof of Claim \ref{N2upper}]
 We prove this claim by contradiction. Suppose that $\sum_{i=1}^5 n'(V_i)<16$. Then, there is a non-strong partition $\tau$ such that $\tau$ is not a useful partition for any $V_i$, $i\in [5]$. In addition, $\tau$ is a useless partition for exactly one partite set, say $V_j$. Now, we have that $\tau$ is a good partition. However, this contradicts Claim~\ref{regularclaim}. 
\end{proof}

Note that $n'(V_i) \leq n(V_i)$. Then by Lemma \ref{lemma rn4} and Lemma \ref{be controlled}, we can derive the following four facts about $n'(V_i)$. 

\begin{fact}\label{fact 1}
    For any regular partite set $V_i$ in $G_{2,5}$, $n'(V_i)\leq 4$.  And if $n'(V_1)=4$, then $V_i$ control\ding{173} another partite set $V_j$.
\end{fact}

\begin{fact}\label{fact 2}
    For any semi-regular partite set $V_i$ in $G_{2,5}$, $n'(V_i)\leq 3$.  And if $n'(V_1)=3$, then $V_i$ control\ding{173} another partite set $V_j$.
\end{fact}

\begin{fact}\label{fact 3}
    For any irregular partite set $V_i$ in $G_{2,5}$, $n'(V_i)\leq 2$. 
\end{fact}

\begin{fact}\label{fact 4}
     If a partite set $V_i$ in $G_{2,5}$ is controlled\ding{173} by two other partite sets, then $n'(V_i) = 0$. In particular, if a regular partite set $V_i$ in $G_{2,5}$ is controlled\ding{173} by another partite set, then $n'(V_i) = 0$. And if a semi-regular partite set $V_i$ in $G_{2,5}$ is controlled\ding{173} by another partite set, then $n'(V_i) \leq 1$.
\end{fact}

Next, we give an upper bound on $\sum_{i = 1}^5 n'(V_i)$.

\begin{claim}\label{lastlemma}
    $\sum^5_{i=1}n'(V_i)\leq 16$
\end{claim}

\begin{proof}[Proof of Claim \ref{lastlemma}]
By above facts, we have that if $n'(V_i)=4$, then $V_i$ is regular and $V_i$ control\ding{173} $V_j$ for some $j$.

\textbf{Case 1: there is only one partite set $V_1$ such that $n'(V_1)=4$.}\\
So, $V_1$ is regular and $V_1$ control\ding{173} $V_j$. By Fact~\ref{fact 3} and Fact~\ref{fact 4}, we have $n'(V_j)\leq 2$. So $\sum^5_{i=1}n'(V_i)\leq 4+3+3+3+n'(V_j)<16$.

\textbf{Case 2: there are two partite sets $V_1,V_2$ such that $n'(V_1)=n'(V_2)=4$.}\\ So, $V_1$ and $V_2$ control\ding{172} each other. If $V_1$ and $V_2$ control\ding{173} the same partite set $V_j$, then by Fact \ref{fact 3}, $n'(V_j)=0$.   So $\sum^5_{i=1}n'(V_i)\leq 4+4+3+3+0<16$. 
 If $V_1$ and $V_2$ control\ding{173} different partite sets $V_k,V_l$, then $n'(V_k)\leq 2$ and $n'(V_l)\leq 2$. So, $\sum^5_{i=1}n'(V_i)\leq 4+4+2+2+3<16$.

\textbf{Case 3: there are three partite sets $V_1,V_2,V_3$ such that $n'(V_1)=n'(V_2)=n'(V_3)=4$.}\\ By Fact \ref{fact 1}, $V_1,V_2,V_3$ control\ding{172} each other and all of them will control\ding{173} a partite set. So, there is a partite set $V_j$ being controlled\ding{173} by two different partite sets. By Fact \ref{fact 3}, we have $n'(V_j)=0$. So  $\sum^5_{i=1}n'(V_i)\leq 4+4+4+3+0<16$.

\textbf{Case 4: there are four partite sets $V_1,V_2,V_3,V_4$ such that $n'(V_1)=n'(V_2)=n'(V_3)=n'(V_4)=4$.}\\ In this case, $V_1,V_2,V_3,V_4$ control\ding{172} each other and all of them control\ding{173} $V_5$. Then $\sum_{i=1}^5 n'(V_i) = 16$.  
\end{proof}

Combine the lower bound and upper bound on $\sum_{i = 1}^5 n'(V_i)$, we have $\sum_{i = 1}^5 n'(V_i) = 16$. In this case, $V_1,V_2,V_3,V_4$ control\ding{172} each other and all of them control\ding{173} $V_5$. Note that this structure is the same as the one shown in Theorem \ref{st 5}. By similar arguments in Section \ref{3}, we have $G_{2,5}$ is isomorphic to $H_0$ which is constructed in Theorem \ref{H0}.

Then $G_{2,5}$ is a regular $2$-balanced $5$-partite tournament. While $G_{2,6}$ is also regular, then for any vertex $u \in V(G_{2,5})$, $u$ is good to $V_6$. Recall that for any vertex $v \in V_6$, we can find $i, j \in [5]$ such that $V_i \subseteq N^{+}(v)$ and $V_j \subseteq N^{-}(v)$.

However, we could provide a st-partition of such $G_{2,6}$ as follows. Now, we change the labels of the vertices of $G_{2,5}$ so that they match those of $H_0$. Let $N^+(v_{4,1})\cap V_6 = \{v_{6,1}\}$. Then both $D[\{v_{i,1}: i\in [6]\}]$ and $D[\{v_{i,2}: i\in [6]\}]$ are strong maximal tournaments and so it is a st-partition.

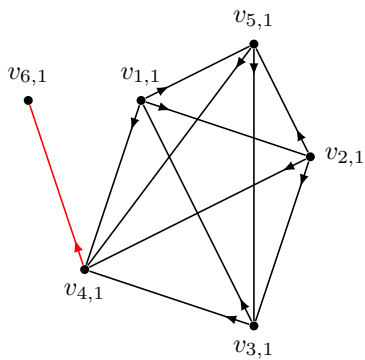
\begin{figure}[h]

    \centering
    \begin{tikzpicture}
    [
    corner/.style={ 
        circle,
        fill,   
        inner sep=1.2pt},	
    arrowline/.style={
        line width=0.6pt,
        postaction = decorate,
        decoration = {markings,
            mark = at position 10pt with \arrow{latex}
        }
    },
    ]

    \node[corner,
    label = {above:$v_{6,1}$}] (v61) at (-2.25,1.5){} ;
    
    \node[corner,
    label = {above:$v_{5,1}$}] (v51) at (0.75,2.25){} ;
    
    \node[corner,
    label = {above:$v_{1,1}$}] (v11) at (-0.75,1.5){} ;
    
    \node[corner,
    label = {right:$v_{2,1}$}] (v21) at (1.5,0.75){} ;

    \node[corner,
    label = {below:$v_{3,1}$}] (v31) at (0.75,-1.5){} ;

    \node[corner,
    label = {below:$v_{4,1}$}] (v41) at (-1.5,-0.75){} ;

    \draw [arrowline] (v11) to (v21);
    \draw [arrowline] (v31) to (v11);
    \draw [arrowline] (v11) to (v41);
    \draw [arrowline] (v11) to (v51);
    
    \draw [arrowline] (v21) to (v41);
    \draw [arrowline] (v31) to (v41);
    \draw [arrowline] (v21) to (v31);
    
    \draw [arrowline] (v21) to (v51);
    \draw [arrowline] (v51) to (v31);
    \draw [arrowline] (v51) to (v41);

    \draw [arrowline,red] (v41) to (v61);

    \end{tikzpicture}
    \caption{An example to explain that there is a st-partition of $G_{2,6}$}\label{Fig:Last}
\end{figure}
\end{proof}

\begin{proof}[Proof of Theorem \ref{mainthm}]
    By Theorem \ref{st27}, Lemma \ref{st25} and Theorem \ref{st26}, we have $ST(2) = 6$.
\end{proof}

\section*{Acknowledgement}
This work was partially supported by the National Natural Science Foundation of China (No. 12161141006), the Natural Science Foundation of Tianjin (No. 20JCJQJC00090), and the Fundamental Research Funds for the Central Universities, Nankai University (No. 63231193).

\bibliographystyle{plain}
\bibliography{ref.bib}
\label{key}
\end{document}